\documentclass[a4paper,11pt]{article}

\usepackage[margin=1in]{geometry}
\usepackage{amsmath,amssymb,amsthm,mathtools}
\usepackage{microtype}
\usepackage{indentfirst}
\usepackage{cite}
\usepackage[hidelinks,pdftitle={Induced-Count Discrepancy in Uniform Hypergraphs},pdfauthor={Chenglong Yu and Qinghou Zeng}]{hyperref}
\numberwithin{equation}{section}

\theoremstyle{plain}
\newtheorem{theorem}{Theorem}[section]
\newtheorem{lemma}[theorem]{Lemma}

\newtheorem{corollary}[theorem]{Corollary}

\theoremstyle{remark}

\newcommand{\E}{\mathbb E}
\newcommand{\Pp}{\mathbb P}
\newcommand{\Var}{\operatorname{Var}}
\newcommand{\Cov}{\operatorname{Cov}}
\newcommand{\dist}{\operatorname{dist}}
\newcommand{\cF}{\mathcal F}
\newcommand{\cB}{\mathcal B}
\newcommand{\cI}{\mathcal I}
\newcommand{\HH}{\mathbb H}
\newcommand{\Er}{E_r}
\newcommand{\inner}[2]{\left\langle #1,#2\right\rangle}
\newcommand{\doi}[1]{\href{https://doi.org/#1}{doi:\nolinkurl{#1}}}
\newcommand{\arxiv}[1]{\href{https://arxiv.org/abs/#1}{arXiv:\nolinkurl{#1}}}

\title{Induced-Count Discrepancy in Uniform Hypergraphs}
\author{Chenglong Yu\footnote{Center for Discrete Mathematics, Fuzhou University, Fujian 350003, China. Email: yuchenglong2002@163.com.}}
\date{}

\begin{document}
	\maketitle
	
	\begin{abstract}
		Let $u_k^{(r)}(H,p)$ be the maximum, over all $k$-vertex $r$-uniform hypergraphs $F$, of the absolute difference between the number of induced copies of $F$ in $H$ and its expectation in the binomial random $r$-graph with the same number of vertices as $H$ and edge probability $p$. Let $u_k^{(r)}(n,p)$ be the minimum of $u_k^{(r)}(H,p)$ over all $n$-vertex $r$-uniform hypergraphs, and let $D_r(n,p)$ be the distance from $p\binom nr$ to the nearest integer. For fixed $r\ge3$, $k\ge2r$ and $\varepsilon>0$, the quantity $u_k^{(r)}(n,p)$ has order
		$\max\!\left\{D_r(n,p),p(1-p)\right\}n^{k-r}$
		whenever $p(1-p)n^{\lfloor r/2\rfloor/(\binom{2r-2}{r}-1)}\ge\varepsilon$. Using local moment estimates and switchings in a sparse reservoir, we construct one hypergraph with the nearest-integer edge count that controls all induced orders simultaneously.
	\end{abstract}


	\section{Introduction}
	
	How closely can a deterministic hypergraph reproduce the induced subhypergraph counts of a binomial random hypergraph? Exact agreement is constrained by integrality: the number of edges is an integer, whereas its random-model expectation need not be. At larger orders there is a further obstruction arising from the variance of the random edge count. We study the minimum worst-case deviation compatible with these two obstructions.
	
	An $r$-graph is a simple $r$-uniform hypergraph. Write $v(H)=|V(H)|$, $e(H)=|E(H)|$, and $[n]=\{1,\ldots,n\}$. Let $\HH_{n,p}^{(r)}$ be the random $r$-graph on $[n]$ in which each $r$-set is an edge independently with probability $p$. For $r$-graphs $F$ and $H$, let $N(F,H)$ count the number of induced subhypergraphs of $H$ that are isomorphic to $F$. Thus, if $v(F)=k$, then
	\[
	\E N(F,\HH_{n,p}^{(r)})
	=\frac{(n)_k}{|\operatorname{Aut}(F)|}p^{e(F)}(1-p)^{\binom{k}{r}-e(F)},
	\]
	where $(n)_k=n(n-1)\cdots(n-k+1)$. We consider
	\[
	\begin{aligned}
		u_k^{(r)}(H,p)&=\max_{v(F)=k}\bigl|N(F,H)-\E N(F,\HH_{n,p}^{(r)})\bigr|,\\
		u_k^{(r)}(n,p)&=\min_{v(H)=n}u_k^{(r)}(H,p).
	\end{aligned}
	\]
	The maximum ranges over isomorphism types of $k$-vertex $r$-graphs. We also set $D_r(n,p)=\dist(p\binom nr,\mathbb Z)$. At order $r$ the edge and nonedge discrepancies have the same absolute value, so $u_r^{(r)}(n,p)=D_r(n,p)$. Throughout, the density $p=p(n)$ may vary with $n$.
	
	The proportional-graph problem originated with Janson and Kratochv\'il~\cite{JK1991}. Probabilistic constructions, existence results and a graph Fourier transform were developed in~\cite{Janson1995,JS1992,Karrman1993}. Naves, Pikhurko and Scott~\cite{NPS2018} determined the following quantitative order for ordinary graphs.
	
	\begin{theorem}[Naves, Pikhurko and Scott~\cite{NPS2018}]\label{thm:graph-case}
		Fix $k\ge4$. If $p=p(n)$ satisfies $0<p<1$ and $p(1-p)n^{1/2}\to\infty$, then
		\[
		u_k^{(2)}(n,p)=\Theta_k\!\left(\max\{D_2(n,p),p(1-p)\}n^{k-2}\right).
		\]
	\end{theorem}
	
	They also asked about extensions to uniform hypergraphs. More recently, Sah, Sawhney and Zhu~\cite{SSZ2024} proved a joint local limit theorem for subgraph counts and applied it to existence and enumeration questions for proportional graphs.
	
	The wider study of discrepancy in graphs and hypergraphs includes the work of Erd\H{o}s and Spencer~\cite{ES1971}, Erd\H{o}s, Goldberg, Pach and Spencer~\cite{EGPS1988}, and Bollob\'as and Scott~\cite{BS2006,BS2011,BS2015}. Related questions for random graphs and hypergraphs appear in~\cite{MNS2015}. Graded and relative discrepancy are studied in~\cite{CHZ2025,CW2026,LLTY2025}, while connections with bisection, MaxCut and eigenvalues appear in~\cite{RST2026,RT2026}. Our focus here is the discrepancy of induced isomorphism counts.
	
	The integrality and variance arguments of~\cite{NPS2018} extend to arbitrary uniformity. They give the following lower bounds, with no density restriction.
	
	\begin{theorem}\label{thm:lower-main}
		Fix $r\ge2$. For every fixed $k\ge r$ there is $c_{r,k}>0$ such that, for all $n\ge k$ and $0<p<1$,
		\[
		u_k^{(r)}(n,p)\ge c_{r,k}D_r(n,p)n^{k-r}.
		\]
		If $k\ge2r$, there is also $c'_{r,k}>0$ such that, under the same conditions,
		\[
		u_k^{(r)}(n,p)\ge c'_{r,k}\max\{D_r(n,p),p(1-p)\}n^{k-r}.
		\]
	\end{theorem}
	
	For $r\ge3$, put
	\[
	A_r=\frac{\binom{2r-2}{r}-1}{\lfloor r/2\rfloor}.
	\]
	Our main upper bound attains the same scale whenever $p(1-p)n^{1/A_r}$ is bounded below by a positive constant.
	
	\begin{theorem}\label{thm:general-upper}
		Fix $r\ge3$ and $\varepsilon>0$. There are $n_0=n_0(r,\varepsilon)$ and constants $C_{r,j,\varepsilon}>0$ for all integers $j\ge r+1$ such that, whenever $n\ge n_0$, $0<p<1$, and $p(1-p)n^{1/A_r}\ge\varepsilon$,
		there exists an $n$-vertex $r$-graph $H$ satisfying $\left|e(H)-p\binom nr\right|=D_r(n,p)$
		and, simultaneously for $r+1\le j\le n$,
		\[
		u_j^{(r)}(H,p)\le C_{r,j,\varepsilon}
		\max\{D_r(n,p),p(1-p)\}n^{j-r}.
		\]
	\end{theorem}
	
	The hypergraph $H$ in Theorem~\ref{thm:general-upper} is independent of $j$. The constants $C_{r,j,\varepsilon}$ are finite and may depend on $j$, while the threshold $n_0$ is chosen only in terms of $r$ and $\varepsilon$. Hence, for any fixed $k\ge2r$, Theorems~\ref{thm:lower-main} and~\ref{thm:general-upper} immediately give the order of $u_k^{(r)}(n,p)$ in the stated density range.
	
	\begin{corollary}\label{cor:main-order}
		Fix $r\ge3$, $k\ge2r$ and $\varepsilon>0$. Under the density condition in Theorem~\ref{thm:general-upper},
		\[
		u_k^{(r)}(n,p)=\Theta_{r,k,\varepsilon}\!\left(
		\max\{D_r(n,p),p(1-p)\}n^{k-r}\right).
		\]
		For $r<k<2r$, the same conclusion holds whenever $D_r(n,p)\ge p(1-p)$.
	\end{corollary}
	
	The proof uses the centered-switching framework of Naves, Pikhurko and Scott~\cite{NPS2018}, which builds on~\cite{JS1992}. The centered expansion of induced counts and the exact second-moment identities are standard $r$-uniform extensions of the graph argument. The main new ingredients are a fourth-moment bound that yields $A_r$, useful local switching directions, a simultaneous treatment of all centered types, and a sparse reservoir that keeps switching errors small. Together, these allow one hypergraph to work for every induced order.
	Thus the lower-bound reductions, the centered expansion, and the basic switching philosophy are inherited from~\cite{NPS2018}; the overlap exponent, the uniform local moment bound, the simultaneous all-order preparation, and the multicodegree reservoir estimates are proved here in the $r$-uniform setting.
	
	Section~\ref{sec:lower} proves the lower bounds. Section~\ref{sec:upper} develops the centered statistics and local moment estimates, constructs the reservoir, and proves the upper bound.
	
	\section{Lower bounds}\label{sec:lower}
	
	The proof has two parts. A double-counting identity transfers any discrepancy at order $j$ to order $j+1$. This lifts the edge-count integrality obstruction to every larger order. A separate identity for the square of the edge count gives the variance obstruction at order $2r$, after which the same transfer applies. We give the details of the $r$-uniform versions of the arguments in~\cite{NPS2018}.
	
	\begin{lemma}\label{lem:order-lifting}
		Fix $r\ge2$ and $r\le j<n$. There are constants $c_{r,j},C_{r,j}>0$ such that every $n$-vertex $r$-graph $H$ and every $0<p<1$ satisfy
		\[
		u_{j+1}^{(r)}(H,p)\ge c_{r,j}n\,u_j^{(r)}(H,p),
		\qquad
		u_j^{(r)}(H,p)\le \frac{C_{r,j}}n u_{j+1}^{(r)}(H,p).
		\]
	\end{lemma}
	
	\begin{proof}
		Let $F$ be an $r$-graph attaining $u_j^{(r)}(H,p)$. Count pairs $(A,x)$ with $H[A]\cong F$ and $x\in V(H)\setminus A$. Counting first by $A$, and then by the isomorphism type induced on $A\cup\{x\}$, gives
		\[
		(n-j)N(F,H)=\sum_{v(F')=j+1}N(F,F')N(F',H).
		\]
		Here $N(F,F')$ is the number of $j$-vertex subsets of $F'$ that induce a copy of $F$. The same identity holds in expectation in the binomial model. Subtracting and taking absolute values therefore yields
		\begin{equation}\label{eq:lifting}
			\begin{aligned}
				(n-j)u_j^{(r)}(H,p)
				&\le\sum_{v(F')=j+1}N(F,F')u_{j+1}^{(r)}(H,p)\\
				&\le 2^{\binom{j+1}{r}}(j+1)u_{j+1}^{(r)}(H,p).
			\end{aligned}
		\end{equation}
		Here there are at most $2^{\binom{j+1}{r}}$ types $F'$, and $N(F,F')\le j+1$. Since $n-j\ge n/(j+1)$, even in the boundary case $n=j+1$, the first displayed inequality follows. The second is the same estimate rearranged.
	\end{proof}
	
	\begin{lemma}\label{lem:variance-obstruction}
		For each fixed $r\ge2$ there is $c_r>0$ such that, for every $n\ge2r$, every $0<p<1$, and every $n$-vertex $r$-graph $H$,
		\[
		u_{2r}^{(r)}(H,p)\ge c_rp(1-p)n^r.
		\]
	\end{lemma}
	
	\begin{proof}
		Put $q=p(1-p)$ and $N=\binom nr$. Suppose, for a contradiction, that $u_{2r}^{(r)}(H,p)<\tau qn^r$, where $\tau>0$ will be chosen in terms of $r$. Iterating Lemma~\ref{lem:order-lifting} backwards gives
		\[
		u_j^{(r)}(H,p)\le C_r\tau qn^{j-r}\qquad(r\le j\le2r).
		\]
		In particular, the edge error $\delta=e(H)-pN$ satisfies $|\delta|\le C_r\tau q$. If $X=e(\HH_{n,p}^{(r)})$, then $X$ is binomial with mean $pN$ and variance $qN$, so
		\[
		e(H)^2=(pN)^2+2pN\delta+\delta^2,
		\qquad \E X^2=(pN)^2+qN.
		\]
		It follows that
		\begin{equation}\label{eq:variance-gap}
			|e(H)^2-\E X^2|\ge qN-2pN|\delta|-\delta^2\ge\tfrac12qN
		\end{equation}
		for sufficiently small $\tau=\tau(r)$. Indeed, after division by $qN$, the two error terms are at most $2C_r\tau p$ and $C_r^2\tau^2q/N$, respectively. Thus one choice of $\tau$ works for all $n\ge2r$ and $0<p<1$.
		
		On the other hand, $e(H)^2$ counts ordered pairs of edges, allowing the two edges to coincide. Their union has between $r$ and $2r$ vertices. Grouping these pairs by the induced hypergraph on their union gives
		\begin{equation}\label{eq:edge-square-expansion}
			e(H)^2=\sum_F\alpha_FN(F,H),
		\end{equation}
		where
		\[
		\alpha_F=\bigl|\{(e_1,e_2)\in E(F)^2:e_1\cup e_2=V(F)\}\bigr|.
		\]
		The sum is over isomorphism types of $r$-graphs without isolated vertices and with $r\le v(F)\le2r$. Every ordered edge pair is counted exactly once, even when its union induces additional edges. The same coefficients therefore give $\E X^2=\sum_F\alpha_F\E N(F,\HH_{n,p}^{(r)})$.
		
		Subtracting these identities and using the assumed bounds at orders $r,\ldots,2r$, we obtain
		\[
		|e(H)^2-\E X^2|
		\le\sum_F\alpha_Fu_{v(F)}^{(r)}(H,p)
		\le C'_r\tau qn^r.
		\]
		There are only finitely many types in this sum, depending on $r$, and $n^{v(F)-r}\le n^r$. Since $N\ge c_rn^r$ for $n\ge2r$, the last inequality contradicts~\eqref{eq:variance-gap} when $\tau$ is small enough. This proves the lemma.
	\end{proof}
	
	\begin{proof}[Proof of Theorem~\ref{thm:lower-main}]
		For every $n$-vertex $r$-graph $H$,
		\[
		u_r^{(r)}(H,p)=\left|e(H)-p\binom nr\right|\ge D_r(n,p).
		\]
		Iterating Lemma~\ref{lem:order-lifting} proves the first assertion. Lemma~\ref{lem:variance-obstruction} gives the $p(1-p)$ lower bound at order $2r$, and another iteration of Lemma~\ref{lem:order-lifting} gives it at every order $k\ge2r$. Combining the two lower bounds, decreasing the constant if necessary, and then minimizing over $H$ proves the theorem.
	\end{proof}
	
	\section{Upper bound for uniform hypergraphs}\label{sec:upper}
	
	We prove Theorem~\ref{thm:general-upper} by starting from a random $r$-graph and changing $O_{r,\varepsilon}(\sqrt p\,n^{r/2})$ edges. Centered subhypergraph statistics express every induced-count discrepancy in a form suited to these changes. We correct the statistics on fewer than $2r$ vertices whose random fluctuations can exceed the required error scale, using a finite set of local switching directions. A sparse reservoir keeps these directions stable and controls the accumulated error in all larger statistics.
	
	Write $q=p(1-p)$. Complementation replaces $p$ by $1-p$, permutes the induced types, and preserves $D_r(n,p)$. We may therefore assume $0<p\le\tfrac12$, so $p/2\le q\le p$. Until a realization is fixed, let $H\sim\HH_{n,p}^{(r)}$, and put $T_0=\sqrt p\,n^{r/2}$. For an $r$-graph $G$ and $1\le a<r$, write $\Delta_a(G)$ for the maximum number of edges containing a fixed $a$-set. Constants with a hypergraph subscript may depend on its fixed isomorphism type; other implicit constants in this section may depend on $r$ and $\varepsilon$.
	
	\subsection{Centered statistics}\label{subsec:centered}
	
	For each $r$-set $e\subseteq[n]$, put $X_e=I_H(e)-p$, where $I_H(e)$ is its edge indicator. The variables $X_e$ are independent, with $\E X_e=0$, $\E X_e^2=q$ and $|X_e|\le1$. For a nonempty $r$-graph $J$ without isolated vertices, define
	\[
	S(J,H)=\sum_{J'\cong J}\prod_{e\in E(J')}X_e.
	\]
	The sum is over all unlabeled copies on $[n]$, regardless of whether their edges belong to $H$. Thus $S(J,H)$ is a sum over edge patterns, whereas $N(F,H)$ counts induced copies. Let $\cF_j$ be the finite family of such types with at most $j$ vertices, and set $\cF_\infty=\bigcup_{j\ge r}\cF_j$. Let $\Er$ denote the unique one-edge $r$-graph. Then $S(\Er,H)=e(H)-p\binom nr$.
	
	For a $j$-vertex $r$-graph $F$, write its induced-copy count as
	\[
	N(F,H)=\frac1{|\operatorname{Aut}(F)|}
	\sum_{\phi:V(F)\hookrightarrow[n]}
	\prod_{e\in E(F)}I_H(\phi(e))
	\prod_{e\in\binom{V(F)}r\setminus E(F)}(1-I_H(\phi(e))).
	\]
	Substituting $I_H(e)=p+X_e$, the constant term is the binomial expectation. For every other monomial, delete the vertices outside its nonempty edge support and group by the resulting type $J\in\cF_j$. As in~\cite{NPS2018}, this gives
	\begin{equation}\label{eq:centered-expansion}
		N(F,H)-\E N(F,\HH_{n,p}^{(r)})
		=\sum_{J\in\cF_j}a_{J,F}(n,p)S(J,H),\qquad
		|a_{J,F}(n,p)|=O_{r,j}(n^{j-v(J)}).
	\end{equation}
	Indeed, a fixed copy of $J$ has $O_{r,j}(n^{j-v(J)})$ extensions to the vertices of an induced-copy term. Each extension contributes a bounded number of products of $p$ and $1-p$, with signs. The coefficient bound is therefore uniform in $0\le p\le1$.
	
	For an $r$-set $t$, define the discrete derivative
	\[
	S_t(J,H)=S(J,H\cup\{t\})-S(J,H\setminus\{t\})
	=\sum_{\substack{J'\cong J\\t\in E(J')}}
	\prod_{e\in E(J')\setminus\{t\}}X_e.
	\]
	Here adjoining or deleting $t$ forces its indicator to one or zero. The derivative does not depend on $I_H(t)$. Adding $t$ changes $S(J,H)$ by $S_t(J,H)$, and removing it changes $S(J,H)$ by $-S_t(J,H)$.
	
	Distinct square-free monomials in centered independent variables are orthogonal: a variable in the symmetric difference of their supports occurs only once in their product. Since $J$ has no isolated vertices, its edge support determines its copy. The same is true for a derivative monomial after adjoining $t$. Writing $f=v(J)$ and $m=e(J)$, we obtain
	\begin{equation}\label{eq:centered-moments}
		\begin{aligned}
			\E S(J,H)^2&=\frac{(n)_f}{|\operatorname{Aut}(J)|}q^m\le C_Jp^mn^f,\\
			\E S_t(J,H)^2&=\frac{r!m(n-r)_{f-r}}{|\operatorname{Aut}(J)|}q^{m-1}
			\le C'_Jp^{m-1}n^{f-r}\qquad(f>r).
		\end{aligned}
	\end{equation}
	For the second identity, choose the edge of $J$ mapped to $t$, map its vertices in $r!$ ways, and inject the remaining vertices. Both statistics have mean zero; the one-edge derivative, excluded from the second line, equals one. These are the uniform-hypergraph versions of the centered moment identities in~\cite{NPS2018}.
	
	We prepare bounds for every type at once, so that the final graph will work at all induced orders. Choose positive constants $A_J$ with $\sum_{J\in\cF_\infty}C_J/A_J^2\le0.005$. Chebyshev's inequality and a countable union bound give, with probability at least $0.995$,
	\begin{equation}\label{eq:initial-random-bound}
		|S(J,H)|\le A_Jp^{e(J)/2}n^{v(J)/2}\qquad(J\in\cF_\infty).
	\end{equation}
	Types with more than $n$ vertices contribute zero. In particular, the constants are chosen independently of any target induced order.
	
	We also exclude $r$-sets whose derivatives are too large. Given $\kappa>0$, choose positive constants $L_J$ for all $v(J)>r$ with $\sum_{v(J)>r}C'_J/L_J^2\le\kappa/400$, and define
	\begin{equation}\label{eq:bad-derivatives}
		\cB=\left\{t:\ |S_t(J,H)|>L_Jp^{(e(J)-1)/2}n^{(v(J)-r)/2}
		\text{ for some }J\in\cF_\infty,\ v(J)>r\right\}.
	\end{equation}
	Another countable union bound gives $\Pp(t\in\cB)\le\kappa/400$. The event $t\in\cB$ is independent of $I_H(t)$, since every derivative in its definition is. Hence
	\[
	\E|\cB\cap E(H)|\le\frac{\kappa}{400}p\binom nr,
	\qquad
	\E|\cB\setminus E(H)|\le\frac{\kappa}{400}(1-p)\binom nr.
	\]
	Markov's inequality shows that the bounds
	\begin{equation}\label{eq:bad-set-size}
		|\cB\cap E(H)|\le\kappa pn^r,\qquad
		|\cB\setminus E(H)|\le\kappa n^r
	\end{equation}
	and~\eqref{eq:initial-random-bound} hold together with probability at least $0.99$. We will choose $\kappa$ after obtaining the local directions below.
	
	\subsection{Local directions}\label{subsec:local}
	
	Let $M\ge1$ be a parameter to be fixed in the construction. Call a type $F\in\cF_\infty$ with $r<v(F)<2r$ \emph{active} if
	\begin{equation}\label{eq:activity}
		n^{2r-v(F)}p^{e(F)-2}\ge M,
	\end{equation}
	and let $\mathcal L$ be the active family. Its size $D=|\mathcal L|$ is bounded in terms of $r$. The two-edge type on $r+1$ vertices is active whenever $n^{r-1}\ge M$, so $\mathcal L$ is nonempty for sufficiently large $n$. The activity condition selects the statistics that may need direct correction. For an inactive type with $f$ vertices and $m$ edges, the initial bound already has the required scale, because
	\begin{equation}\label{eq:inactive-scale}
		p^{m/2}n^{f/2}
		=pn^{f-r}\bigl(n^{2r-f}p^{m-2}\bigr)^{1/2}
		<\sqrt M\,pn^{f-r}.
	\end{equation}
	The same activity condition will make each active switching direction stable under the reservoir changes.
	
	For $F\in\mathcal L$, set $d_F=p^{(e(F)-1)/2}n^{(v(F)-r)/2}$ and define
	\begin{equation}\label{eq:local-vector}
		W_t(G)=\left(\frac{S_t(F,G)}{d_F}\right)_{F\in\mathcal L},
	\end{equation}
	\begin{equation}\label{eq:global-R}
		R(G)=\left(\frac{S(F,G)}{d_F}\right)_{F\in\mathcal L}.
	\end{equation}
	The denominators are the standard-deviation scales in~\eqref{eq:centered-moments}. They are fixed throughout the switching procedure. The derivative identity becomes
	$R(G\cup\{t\})-R(G\setminus\{t\})=W_t(G)$.
	We write $W_t=W_t(H)$ for the random local vector. Its law is independent of $t$ by exchangeability, and~\eqref{eq:initial-random-bound} gives
	\begin{equation}\label{eq:R-initial}
		|S(\Er,H)|+\|R(H)\|\le C T_0.
	\end{equation}
	The constant can be chosen for all active families, independently of $M$, since there are finitely many types with fewer than $2r$ vertices.
	
	\begin{lemma}\label{lem:overlap-exponent}
		For $s\ge0$, put $g(s)=\binom{r+s}{r}-1$. If $r\ge3$, $1\le d\le r-1$ and $1\le j\le d-1$, then
		\[
		\frac{\min\{g(d-1),g(2j)/2\}}j\le A_r.
		\]
		Moreover, $g(s)/s$ is nondecreasing for positive integers $s$, and $\frac{g(d)}d\le A_r$ whenever $d\le\lfloor r/2\rfloor$.
		
	\end{lemma}
	
	\begin{proof}
		Vandermonde's identity gives
		\[
		g(s)=\sum_{i\ge1}\binom ri\binom si.
		\]
		For each $i\ge1$, $\binom si/s=\binom{s-1}{i-1}/i$ is nondecreasing in $s$, so $g(s)/s$ is nondecreasing.
		
		If $2j\le d-1$, then
		\[
		\frac{\min\{g(d-1),g(2j)/2\}}j
		\le \frac{g(2j)}{2j}\le \frac{g(d-1)}{d-1}
		\le \frac{g(d-1)}{\lfloor(d+1)/2\rfloor}.
		\]
		If $2j>d-1$, then $j\ge\lfloor(d+1)/2\rfloor$, and the same final upper bound follows directly from the first term in the minimum. Thus it remains to bound $\frac{g(d-1)}{\lfloor(d+1)/2\rfloor}$ by $A_r$. Put $h(d)=g(d-1)/\lfloor(d+1)/2\rfloor$. If $d$ changes from $2a-1$ to $2a$, the denominator is unchanged and the numerator increases. The only nontrivial step is therefore from $d=2a$ to $d=2a+1$, where the denominator changes from $a$ to $a+1$. In that case
		\[
		a\,g(2a)-(a+1)g(2a-1)
		=(r/2-1)g(2a-1)+r/2\ge0.
		\]
		Here we used
		\[
		g(2a)-g(2a-1)=\binom{r+2a-1}{r-1}
		=\frac r{2a}\binom{r+2a-1}r
		=\frac r{2a}\bigl(g(2a-1)+1\bigr).
		\]
		Hence $h$ is nondecreasing, and
		\[
		\frac{g(d-1)}{\lfloor(d+1)/2\rfloor}
		\le \frac{g(r-2)}{\lfloor r/2\rfloor}=A_r.
		\]
		Finally, if $d\le\lfloor r/2\rfloor$, monotonicity of $g(s)/s$ gives
		\[
		\frac{g(d)}d\le\frac{g(\lfloor r/2\rfloor)}{\lfloor r/2\rfloor}
		\le\frac{g(r-2)}{\lfloor r/2\rfloor}=A_r.
		\]
	\end{proof}
	
	\begin{lemma}\label{lem:local-moments}
		Assume that $p(1-p)n^{1/A_r}\ge\varepsilon$ and $0<p\le1/2$. There are constants $C=C(r,\varepsilon)$ and $\lambda_0=\lambda_0(r,\varepsilon)>0$ such that, for every $M\ge1$ and every corresponding nonempty active family $\mathcal L$, the local vector $W_t$ satisfies
		\begin{equation}\label{eq:local-law}
			\E\|W_t\|^4\le C,\qquad
			\lambda_{\min}(\Cov W_t)\ge\lambda_0>0,
		\end{equation}
		for all sufficiently large $n$, with the threshold depending only on $r$ and $\varepsilon$.
	\end{lemma}
	
	\begin{proof}
		Since $A_r\ge3$, the density hypothesis implies
		\begin{equation}\label{eq:gen-density}
			p\ge\varepsilon n^{-1/A_r},\qquad
			np^{A_r}\ge\varepsilon^{A_r},\qquad
			p^2n\longrightarrow\infty,\qquad pn\gg\log n.
		\end{equation}
		
		For distinct types $F,G$, their derivative monomials have different supports. Indeed, a derivative monomial records the edge support of a rooted copy with the root edge $t$ removed. Adjoining $t$ recovers the full edge support, and since the types have no isolated vertices this edge support also recovers the vertex set of the rooted copy. Equal supports in two coordinates would therefore give the same rooted labeled hypergraph and force $F=G$. Thus $\Cov W_t$ is diagonal, and its diagonal entry for a type $F$ with $f$ vertices and $m$ edges is exactly
		\[
		\Var\left(\frac{S_t(F,H)}{d_F}\right)
		=\frac{r!m}{|\operatorname{Aut}(F)|}
		\frac{(n-r)_{f-r}}{n^{f-r}}(1-p)^{m-1}.
		\]
		Because $1-p\ge1/2$ and the number of types is finite, these entries have a common positive lower bound for sufficiently large $n$.
		
		Fix an active type $F$ and put $d=f-r$. In the expansion of $S_t(F,H)^4$, consider four rooted copies of $F$ containing $t$. Each contributes a monomial with $m-1$ variables. A quadruple contributes only if every variable occurs at least twice; call such a quadruple surviving. Since $F$ has no isolated vertices, every vertex outside $t$ in a surviving quadruple belongs to at least two of the four rooted copies. Indeed, a vertex belonging to only one rooted copy lies in a non-root input edge that cannot appear in any other copy, giving a variable of multiplicity one.
		
		Let $u$ be the number of vertices outside $t$ in the union of the four rooted copies. Since the four outside vertex sets have total size $4d$ and each vertex is used at least twice, $u\le2d$. If $u=2d$, then each outside vertex belongs to exactly two copies. Every input edge contains an outside vertex, because the only $r$-set contained in $t$ is the root edge itself. Thus a variable occurring three or four times would contain an outside vertex lying in three or four rooted copies, which is impossible in the case $u=2d$. Survival therefore forces every variable to occur exactly twice. These terms contribute $O_F(n^{2d}p^{2(m-1)})$.
		
		If $u=d$, then all four outside vertex sets are equal. There are $O_F(n^d)$ such quadruples, and the product contains at least the $m-1$ distinct variables coming from one copy. Since $\E|X_e|^a\le2p$ for $a\ge2$, their total contribution is $O_F(n^dp^{m-1})$.
		
		It remains to treat $d<u<2d$. Write $u=2d-j$, where $1\le j\le d-1$, and let $b$ be the number of distinct variables in the product. Put $\ell=2(m-1)-b$. For fixed $j$ and $\ell$, the number of such quadruples is $O_F(n^{2d-j})$: after choosing the outside union, there are only $O_F(1)$ choices for the four outside vertex subsets of size $d$ and for the four rooted embeddings of the fixed type $F$. In a surviving quadruple every variable has multiplicity at least two, and $\E|X_e|^a\le2p$ for every $a\ge2$. Hence the absolute expectation of the product is $O_F(p^b)=O_F(p^{2(m-1)-\ell})$. After division by $d_F^4=p^{2(m-1)}n^{2d}$, all quadruples with these values of $j$ and $\ell$ contribute $O_F(n^{-j}p^{-\ell})$.
		
		We now bound $\ell$. Let $b_2,b_3,b_4$ be the numbers of variables occurring exactly two, three and four times. Since
		\[
		2b_2+3b_3+4b_4=4(m-1),\qquad b=b_2+b_3+b_4,
		\]
		we have $2\ell=b_3+2b_4$. Write $g(s)=\binom{r+s}{r}-1$ for the number of possible non-root input edges on $t$ and $s$ additional vertices.
		
		First, because $u>d$, two rooted copies have different outside vertex sets. Their outside intersection has size at most $d-1$, so the two corresponding monomials share at most $g(d-1)$ variables: a shared non-root input edge must be contained in $t$ together with these common outside vertices, and there are at most $\binom{r+d-1}{r}-1=g(d-1)$ such $r$-sets. These two monomials already contain at least $2(m-1)-g(d-1)$ distinct variables, and adding the other two monomials can only enlarge the support. Therefore $b\ge2(m-1)-g(d-1)$, or equivalently $\ell\le g(d-1)$.
		
		Second, create formal vertices as follows. For each outside vertex belonging to three copies, create one formal copy; for each outside vertex belonging to all four copies, create two distinguishable formal copies; vertices belonging to exactly two copies create none. The total number of formal vertices is
		\[
		\sum_y(\text{multiplicity of }y-2)=4d-2(2d-j)=2j.
		\]
		Keep the vertices of $t$ unchanged. A variable occurring three times is mapped to the $r$-set obtained by replacing each outside vertex by its formal copy, using the first formal copy when two are available. A variable occurring four times is mapped to two $r$-sets, using all first formal copies and all second formal copies of its outside vertices. These assigned $r$-sets are distinct. If two assigned sets had different projections to the original vertices, they would clearly be distinct; if their projections were equal, they would correspond to the same input edge, and the only possible duplication is the two images of one fourfold variable, which are distinguished by first versus second formal copies. None is the root edge $t$. Thus $b_3+2b_4\le g(2j)$, so $2\ell\le g(2j)$. Together,
		\[
		\ell\le\min\{g(d-1),g(2j)/2\}.
		\]
		
		Combining the cases $u=2d$, $u=d$ and $d<u<2d$, and absorbing the finite sum over possible $\ell$, we obtain
		\begin{equation}\label{eq:gen-fourth-moment}
			\frac{\E S_t(F,H)^4}{d_F^4}
			\le C_F\left(1+
			\sum_{j=1}^{d-1}\frac1{n^jp^{\min\{g(d-1),g(2j)/2\}}}
			+\frac1{n^dp^{m-1}}\right).
		\end{equation}
		The sum is empty when $d=1$.
		
		By Lemma~\ref{lem:overlap-exponent}, $\min\{g(d-1),g(2j)/2\}\le A_rj$. Since $0<p\le1$, this gives
		\[
		n^jp^{\min\{g(d-1),g(2j)/2\}}
		\ge n^jp^{A_rj}
		=(np^{A_r})^j
		\ge\varepsilon^{A_rj},
		\]
		so every denominator in the sum in~\eqref{eq:gen-fourth-moment} is bounded below uniformly.
		
		For the last term, note first that $m-1\le g(d)$, since all non-root input edges of a rooted copy lie inside $t$ together with its $d$ outside vertices. If $d\le\lfloor r/2\rfloor$, then
		\[
		\frac{m-1}{d}\le\frac{g(d)}d\le A_r,
		\]
		so $m-1\le A_rd$. Since $0<p\le1$, this implies $n^dp^{m-1}\ge(np^{A_r})^d\ge\varepsilon^{A_rd}$. If $d>\lfloor r/2\rfloor$, activity and $2d-r\ge1$ give
		\[
		n^dp^{m-1}
		=\bigl(n^{r-d}p^{m-2}\bigr)n^{2d-r}p
		\ge Mnp.
		\]
		This is at least one for sufficiently large $n$, uniformly in $M\ge1$. Thus~\eqref{eq:gen-fourth-moment} is bounded in terms of $F,\varepsilon$. The possible active families are subfamilies of the finite set of types with $r<v(F)<2r$, so the dimensions and all fourth-moment constants are bounded in terms of $r$ and $\varepsilon$. Since $\|W_t\|^4\le D\sum_{F\in\mathcal L}|S_t(F,H)/d_F|^4$, this proves~\eqref{eq:local-law}, with constants independent of the activity parameter.
	\end{proof}
	
	We turn these moments into a finite set of useful directions. The following finite spanning directions lemma supplies small balls of positive probability whose centers affinely span the whole space. In a paired switching, differences between these centers will change the vector $R$ while preserving the edge count.
	
	\begin{lemma}\label{lem:affine-spread}
		For every integer $D\ge0$ and $\lambda,C>0$ there are $c,R_{\rm sp}>0$ such that, for each $\delta_*>0$, there are $0<\delta\le\delta_*$ and $\eta>0$ with the following property. Every centered $X\in\mathbb R^D$ satisfying $\lambda_{\min}(\Cov X)\ge\lambda$ and $\E\|X\|^4\le C$ admits $z_0,\ldots,z_D\in B(0,R_{\rm sp})$ such that
		\[
		\Pp(X\in B(z_i,\delta))\ge\eta\quad(0\le i\le D),\qquad
		\max_{1\le i\le D}|\inner{x}{z_i-z_0}|\ge c\|x\|
		\quad(x\in\mathbb R^D).
		\]
		When $D=0$, the covariance and spanning assertions are void. For $D\ge1$, the centers satisfy $\|z_i-z_j\|\ge(\lambda/2)^{1/2}$ for $i\ne j$. The constants $c,R_{\rm sp}$ depend only on $D,\lambda,C$ and are independent of $\delta_*$.
	\end{lemma}
	
	\begin{proof}
		If $D=0$, take $z_0=0$ and any sufficiently small $\delta$; then the assertion is immediate.
		Put $r_*=(\lambda/2)^{1/2}$ and $\delta=\frac12\min\{\delta_*,r_*/16\}$. For $R_0=(2C)^{1/4}$, Markov's inequality gives $\Pp(\|X\|\le R_0)\ge1/2$. Cover $B(0,R_0)$ by finitely many balls of radius $\delta$, with centers in $B(0,R_0)$. One covering ball has probability bounded below in terms of $D,R_0,\delta$; choose its center as $z_0$.
		
		Suppose $z_0,\ldots,z_j$ have been selected, with $j<D$, and let $L_j$ be their affine span. Choose a unit vector $u$ orthogonal to its direction space and put $Y=\inner{u}{X-z_0}$. Then $\E Y^2\ge\lambda$ and $\E Y^4\le8(C+\|z_0\|^4)\le24C$. Since the hyperplane $\{x:\inner{u}{x-z_0}=0\}$ contains $L_j$, the Paley--Zygmund inequality gives
		\[
		\Pp(\dist(X,L_j)\ge r_*)
		\ge\Pp(Y^2\ge\tfrac12\E Y^2)
		\ge\frac{\lambda^2}{96C}=: \alpha.
		\]
		Choose $R_{\rm sp}\ge R_0$, depending only on $C,\alpha$, so that $\Pp(\|X\|>R_{\rm sp})\le\alpha/2$. The set
		\[
		\{x\in B(0,R_{\rm sp}):\dist(x,L_j)\ge r_*\}
		\]
		has probability at least $\alpha/2$. A finite $\delta$-net with centers in this set has cardinality bounded in terms of $D,R_{\rm sp},\delta$. One of its balls has uniformly positive probability. Choose its center as $z_{j+1}$; then $\dist(z_{j+1},L_j)\ge r_*$.
		
		Iteration gives the centers and a common lower mass bound $\eta$. Every new center is at distance at least $r_*$ from all previous centers. Let $V$ be the matrix with columns $z_i-z_0$, for $1\le i\le D$. Their successive Gram--Schmidt distances are at least $r_*$, and their norms are at most $2R_{\rm sp}$. Hence
		\[
		|\det V|\ge r_*^D,\qquad
		\sigma_{\min}(V)\ge\frac{r_*^D}{(2R_{\rm sp}\sqrt D)^{D-1}}.
		\]
		The inequality $\max_i|\inner{x}{z_i-z_0}|\ge D^{-1/2}\|V^{\mathsf T}x\|$ proves the spanning assertion, with $c$ independent of $\delta_*$.
	\end{proof}
	
	For the present active family, apply the lemma to $W_t$, choosing $\delta_*$ small enough that $\delta_*\le c/32$ and the balls of radius $2\delta_*$ around the separated centers are disjoint. We obtain
	\begin{equation}\label{eq:spread}
		\Pp(W_t\in B(z_i,\delta))\ge\eta\quad(0\le i\le D),\qquad
		\max_{1\le i\le D}|\inner{x}{z_i-z_0}|\ge c\|x\|,
	\end{equation}
	where $0<\delta\le c/32$, $\|z_i\|\le R_{\rm sp}$, and the balls $B(z_i,2\delta)$ are disjoint. The centers may depend on the active family and on $(n,p)$, but the constants $c,R_{\rm sp},\delta,\eta$ are chosen uniformly over all active families and all $M\ge1$, by taking minima and maxima over the finitely many dimensions in~\eqref{eq:local-law}. Set $\kappa=\eta/(100r!)$, and choose the constants $L_J$ in~\eqref{eq:bad-derivatives} accordingly.
	
	The positive masses in~\eqref{eq:spread} must now be realized simultaneously by many $r$-sets in one hypergraph. Choose functions $\psi_i:\mathbb R^D\to[0,1]$, equal to one on $B(z_i,\delta)$ and zero outside $B(z_i,2\delta)$, with Lipschitz constants at most $1/\delta$. Their covariance is controlled by resampling common inputs.
	
	\begin{lemma}\label{lem:decoupling}
		For disjoint $r$-sets $t,s$, there is a coupling with independent vectors $\widetilde W_t,\widetilde W_s$, having the respective marginal laws of $W_t,W_s$, such that
		\[
		\E\|W_t-\widetilde W_t\|^2+\E\|W_s-\widetilde W_s\|^2=O_r(n^{-1}).
		\]
		Consequently, $|\Cov(\psi_i(W_t),\psi_i(W_s))|=O_{r,\varepsilon}(n^{-1/2})$.
	\end{lemma}
	
	\begin{proof}
		Every input edge of $W_t$ meets $t$ and differs from $t$: its supporting type contains $t$ and has fewer than $2r$ vertices. For disjoint $t,s$, all common inputs therefore lie in
		\[
		\mathcal C=\{e\in\tbinom{[n]}r:e\cap t\ne\varnothing,\ e\cap s\ne\varnothing\}.
		\]
		For each $e\in\mathcal C$, take two copies $X_e^{(t)},X_e^{(s)}$ that are mutually independent over all $e$ and independent of the original variables. Replace these inputs in $W_t$ using the first family and in $W_s$ using the second. The resulting vectors $\widetilde W_t,\widetilde W_s$ have the required marginals and depend on disjoint independent input families.
		
		Let $H^{(e)}$ be obtained by replacing $X_e$ with an independent copy $X'_e$. Conditional on the inputs outside $\mathcal C$, the vectors $W_t$ and $\widetilde W_t$ are independent and identically distributed. Denoting the sigma-field of those outside inputs by $\mathcal U$, conditional Efron--Stein, applied to each coordinate, gives
		\begin{equation}\label{eq:common-input-resampling}
			\begin{aligned}
				\E\|W_t-\widetilde W_t\|^2
				&=2\E\|W_t-\E(W_t\mid\mathcal U)\|^2\\
				&\le\sum_{e\in\mathcal C}\E\|W_t-W_t(H^{(e)})\|^2.
			\end{aligned}
		\end{equation}
		We estimate the squared influence of one input. For a coordinate $F$ with $f$ vertices and $m$ edges, its unnormalized change is
		\[
		S_t(F,H)-S_t(F,H^{(e)})
		=(X_e-X'_e)\sum_{\substack{F'\cong F\\t,e\in E(F')}}
		\prod_{g\in E(F')\setminus\{t,e\}}X_g.
		\]
		Distinct copies give distinct supports in this sum, because adjoining $t$ and $e$ recovers their complete edge sets. The monomials are orthogonal, with squared expectation $q^{m-2}$, and the independent factor $X_e-X'_e$ has squared expectation $2q$.
		
		If $|e\cap t|=a$, there are $O_F(n^{f-2r+a})$ copies containing both $t$ and $e$, or none when $f<2r-a$. After division by $d_F^2$, the squared influence is thus $O_F(n^{a-r})$. There are $O_r(n^{r-a-1})$ common inputs with this intersection size: after choosing their $a$ vertices in $t$, they must contain at least one vertex of the disjoint set $s$. Multiplying these two bounds and summing over $1\le a<r$ yields $O_F(n^{-1})$. The number of coordinates is bounded in terms of $r$, so~\eqref{eq:common-input-resampling} proves the claimed bound for $W_t$. Interchanging $t,s$ proves the other one.
		
		The coupled vectors are independent with the original marginal laws. Subtracting product expectations, using $0\le\psi_i\le1$, and applying the Lipschitz bound and Cauchy--Schwarz gives
		\[
		\begin{aligned}
			|\Cov(\psi_i(W_t),\psi_i(W_s))|
			&\le\E|\psi_i(W_t)-\psi_i(\widetilde W_t)|
			+\E|\psi_i(W_s)-\psi_i(\widetilde W_s)|\\
			&=O_{r,\varepsilon}(n^{-1/2}).
		\end{aligned}
		\]
	\end{proof}
	
	Define the candidate classes
	\[
	E_i^+=\{t\in E(H):W_t\in B(z_i,2\delta)\},\qquad
	E_i^-=\{t\notin E(H):W_t\in B(z_i,2\delta)\}.
	\]
	The superscripts indicate edges available for removal and nonedges available for addition. All classes are disjoint.
	
	\begin{lemma}\label{lem:candidate-concentration}
		For every fixed choice of $M\ge1$ and all sufficiently large $n$, with positive probability the bounds~\eqref{eq:initial-random-bound}, \eqref{eq:bad-set-size} and
		\begin{equation}\label{eq:good-candidates}
			|E_i^+\setminus\cB|\ge\kappa pn^r,\qquad
			|E_i^-\setminus\cB|\ge\kappa n^r\quad(0\le i\le D),\qquad
			\Delta_a(H)=O_r(pn^{r-a})\quad(1\le a<r).
		\end{equation}
		hold simultaneously.
	\end{lemma}
	
	\begin{proof}
		The class sizes dominate the smoothed counts
		\[
		X_i^+=\sum_t I_H(t)\psi_i(W_t),\qquad
		X_i^-=\sum_t(1-I_H(t))\psi_i(W_t).
		\]
		Since $W_t$ is independent of $I_H(t)$, their expectations are at least $\eta p\binom nr$ and $\eta(1-p)\binom nr$, respectively. For disjoint $t,s$, neither root indicator occurs among the inputs of either local vector. Thus $I_H(t),I_H(s)$ are jointly independent of $(W_t,W_s)$, and
		\[
		\Cov(I_H(t)\psi_i(W_t),I_H(s)\psi_i(W_s))
		=p^2\Cov(\psi_i(W_t),\psi_i(W_s)).
		\]
		The analogous identity for nonedges has factor $(1-p)^2$.
		There are $O_r(n^r)$ diagonal pairs, $O_r(n^{2r-1})$ distinct intersecting pairs and $O_r(n^{2r})$ disjoint pairs. Bounding the intersecting-pair covariances by one and using Lemma~\ref{lem:decoupling} for disjoint pairs, we obtain
		\[
		\begin{aligned}
			\Var X_i^+&=O_{r,\varepsilon}(pn^r+n^{2r-1}+p^2n^{2r-1/2})=o(p^2n^{2r}),\\
			\Var X_i^-&=O_{r,\varepsilon}(n^{2r-1}+n^{2r-1/2})=o(n^{2r}).
		\end{aligned}
		\]
		The displayed terms come from diagonal, intersecting-root and disjoint-root pairs, respectively. Here only $p^2n\to\infty$ is used. Chebyshev's inequality, the choice $\kappa=\eta/(100r!)$, and a union bound over the finitely many indices $0\le i\le D$ show that every class has size at least $2\kappa pn^r$ or $2\kappa n^r$, respectively, with probability $1-o(1)$.
		
		For each $a$-set, its degree in $H$ is binomial with mean $p\binom{n-a}{r-a}$. Since $pn\gg\log n$, Chernoff's inequality and a union bound give $\Delta_a(H)=O_r(pn^{r-a})$ for every $1\le a<r$, with probability $1-o(1)$. The initial and bad-set bounds~\eqref{eq:initial-random-bound} and~\eqref{eq:bad-set-size} hold together with probability at least $0.99$. Intersecting this event with the candidate-size and degree events therefore gives positive probability for all sufficiently large $n$. On this event, \eqref{eq:bad-set-size} removes at most $\kappa pn^r$ edges and at most $\kappa n^r$ nonedges from the corresponding candidate classes, so~\eqref{eq:good-candidates} holds.
	\end{proof}
	
	Every candidate outside $\cB$ satisfies the derivative bounds in~\eqref{eq:bad-derivatives} for all $v(J)>r$. All constants in these preparation bounds are independent of $M$. In the construction below, we first choose the switching budget, then fix $M$ to obtain the required stability, and finally select $H$.
	
	\subsection{The switching construction}\label{subsec:construction}
	
	We select a sparse reservoir from the good candidates, correct the edge count, and then reduce the active local vector by pairs of switches. The next lemmas isolate the degree truncation, stability, paired descent and final error estimate.
	
	\begin{lemma}\label{lem:reservoir}
		Fix $K,A_{\rm s}>0$ with $\kappa A_{\rm s}>4K$. If $L=L(r,\varepsilon,K,A_{\rm s})$ is sufficiently large and $H$ satisfies~\eqref{eq:good-candidates}, then, for all sufficiently large $n$, there are subsets $Q_i^\pm\subseteq E_i^\pm\setminus\cB$ such that, writing $Q=\bigcup_i(Q_i^+\cup Q_i^-)$,
		\begin{equation}\label{eq:reservoir}
			\begin{gathered}
				|Q_i^\pm|\ge KT_0,\qquad |Q|=O_{r,\varepsilon,K,A_{\rm s}}(T_0),\qquad
				Q\cap\cB=\varnothing,\\
				\Delta_a(Q)\le1+L\sqrt p\,n^{r/2-a}\qquad(1\le a<r),
			\end{gathered}
		\end{equation}
		and, for every $1\le a<r$,
		\begin{equation}\label{eq:reservoir-pairs}
			\#\{\{t,s\}\subseteq Q:|t\cap s|\ge a\}
			=O_{r,\varepsilon,K,A_{\rm s},L}(pn^{r-a}).
		\end{equation}
	\end{lemma}
	
	\begin{proof}
		Independently retain each good candidate with probability
		\[
		\theta_+=\frac{A_{\rm s}}{\sqrt p\,n^{r/2}}\quad\text{on }E_i^+\setminus\cB,
		\qquad
		\theta_-=\frac{A_{\rm s}\sqrt p}{n^{r/2}}\quad\text{on }E_i^-\setminus\cB.
		\]
		Both probabilities are at most one for large $n$. By~\eqref{eq:good-candidates}, each sampled class has expected size at least $\kappa A_{\rm s}T_0>4KT_0$. Since $e(H)=O_r(pn^r)$, the expected total selected size is $O_{r,\varepsilon,K,A_{\rm s}}(T_0)$. Chernoff's inequality gives at least $2KT_0$ members in every class and $O_{r,\varepsilon,K,A_{\rm s}}(T_0)$ selected edges in total, with probability $1-o(1)$.
		
		For an $a$-set $U$, where $1\le a<r$, let $D_U$ be its selected degree. The numbers of edge and nonedge candidates containing it are $O_r(pn^{r-a})$ and $O_r(n^{r-a})$, respectively, so
		\[
		\E D_U\le O_r(pn^{r-a})\theta_++O_r(n^{r-a})\theta_-
		=O_{r,\varepsilon,K,A_{\rm s}}(\sqrt p\,n^{r/2-a}).
		\]
		Call $U$ crowded if $D_U>1+L\sqrt p\,n^{r/2-a}$. On this event,
		$D_U-1>L\sqrt p\,n^{r/2-a}$. Since the selections are independent,
		$\E[D_U(D_U-1)]\le(\E D_U)^2$, and hence
		\[
		\E\bigl[D_U\mathbf1_{\{U\text{ crowded}\}}\bigr]
		\le\frac{\E[D_U(D_U-1)]}{L\sqrt p\,n^{r/2-a}}
		\le\frac{C_{r,\varepsilon,K,A_{\rm s}}}L\sqrt p\,n^{r/2-a}.
		\]
		Delete every selected edge containing a crowded set of any size $1\le a<r$. The expected number deleted is at most
		\[
		\sum_{a=1}^{r-1}\sum_{U\in\binom{[n]}a}
		\E\bigl[D_U\mathbf1_{\{U\text{ crowded}\}}\bigr]
		\le\frac{C_{r,\varepsilon,K,A_{\rm s}}}L T_0.
		\]
		Choose $L$ so that the last bound is at most $KT_0/8$. By Markov's inequality, with probability at least $3/4$ at most $KT_0/2$ edges are deleted. This event intersects the preceding class-size and total-size events for large $n$. Fix such an outcome, and denote the surviving classes by $Q_i^\pm\subseteq E_i^\pm\setminus\cB$. Every class still has at least $KT_0$ members. The added $1$ in the truncation threshold makes~\eqref{eq:reservoir} meaningful also when $a>r/2$ and the power of $n$ is negative.
		
		These degree bounds also control the total number of intersecting pairs. Write $d_Q(U)$ for the number of edges of $Q$ containing $U$. For $t\in Q$ and $U\subseteq t$, the edge $t$ itself is counted in $d_Q(U)$. Thus, for each $1\le a<r$,
		\[
		\begin{aligned}
			\#\{\{t,s\}\subseteq Q:|t\cap s|\ge a\}
			&\le\frac12\sum_{t\in Q}\sum_{U\in\binom ta}(d_Q(U)-1)\\
			&\le\frac12\binom ra|Q|L\sqrt p\,n^{r/2-a}
			=O_{r,\varepsilon,K,A_{\rm s},L}(pn^{r-a}).
		\end{aligned}
		\]
	\end{proof}
	
	\begin{lemma}\label{lem:local-stability}
		Assume~\eqref{eq:reservoir}. Let $G$ be obtained from $H$ by toggling a subset of $Q$, each edge at most once, and let $t\in Q$ be unused. Then
		\[
		\|W_t(G)-W_t(H)\|=O_{r,\varepsilon,K,L}(M^{-1/2}).
		\]
		In particular, after $K,A_{\rm s},L$ have been fixed, $M$ can be chosen so that
		\begin{equation}\label{eq:local-stability}
			\|W_t(G)-W_t(H)\|\le c/16
		\end{equation}
		for every such $G$ and $t$.
	\end{lemma}
	
	\begin{proof}
		Let $G$ be obtained from $H$ by toggling a subset of $Q$, each edge at most once, and let $t\in Q$ be unused. Fix an active type $F$ with $f=v(F)$ and $m=e(F)$. Telescope the changes of $S_t(F,\cdot)$ over the earlier toggles. An earlier edge $s$ affects this derivative only through copies of $F$ containing both $t$ and $s$. If $|t\cap s|=a$, there are $O_F(n^{f-2r+a})$ such copies when $2r-a\le f$, and none otherwise. Each corresponding monomial changes by a product of centered indicators of absolute value at most one.
		
		Since $f<2r$, a contributing edge $s$ must meet $t$. For a given $a$, the number of other reservoir edges meeting $t$ in exactly $a$ vertices is at most
		\[
		\sum_{U\in\binom ta}(d_Q(U)-1)
		\le\binom ra L\sqrt p\,n^{r/2-a}.
		\]
		Consequently,
		\[
		\begin{aligned}
			\frac{|S_t(F,G)-S_t(F,H)|}{d_F}
			&\le C_F L\sum_{a=2r-f}^{r-1}
			\frac{\sqrt p\,n^{r/2-a}n^{f-2r+a}}
			{p^{(m-1)/2}n^{(f-r)/2}}\\
			&=O_{r,\varepsilon,K,L}\!\left(
			\frac1{\sqrt{n^{2r-f}p^{m-2}}}\right)
			=O_{r,\varepsilon,K,L}(M^{-1/2}).
		\end{aligned}
		\]
		There are only finitely many possible local coordinates, so the coordinate estimates give the asserted norm bound. Choosing $M$ sufficiently large after $K,A_{\rm s},L$ gives~\eqref{eq:local-stability}. The estimate is uniform over the subset and order of earlier toggles.
	\end{proof}
	
	\begin{lemma}\label{lem:paired-descent}
		Assume~\eqref{eq:reservoir} and~\eqref{eq:local-stability}. If $K$ is larger than a constant depending only on the uniform initial and geometric bounds, then there are an $r$-graph $H_*$ and a set $T\subseteq Q$ of toggled $r$-sets such that
		\begin{equation}\label{eq:toggle-budget}
			e(H_*)=m_*,\qquad \|R(H_*)\|=O_{r,\varepsilon}(1),\qquad
			|T|=O_{r,\varepsilon}(T_0),\qquad T\cap\cB=\varnothing,
		\end{equation}
		where $m_*$ is an integer nearest to $p\binom nr$.
	\end{lemma}
	
	\begin{proof}
		We use the two-stage correction of~\cite{NPS2018}. Let $m_*$ be an integer nearest to $p\binom nr$. Remove unused edges from $Q_0^+$ or add unused edges from $Q_0^-$ until the edge count is $m_*$. By~\eqref{eq:R-initial}, this takes $O_{r,\varepsilon}(T_0)$ toggles. Every current local vector used here has norm at most $R_{\rm sp}+2\delta+c/16$, by~\eqref{eq:local-stability}. Hence the resulting graph $H_0$ has $\|R(H_0)\|=O_{r,\varepsilon}(T_0)$, with a constant determined by the initial and geometric bounds.
		
		We now preserve the edge count by pairing one addition with one removal. To move approximately in direction $z_i-z_0$, add an unused edge from $Q_i^-$ and remove one from $Q_0^+$; for the opposite direction, add from $Q_0^-$ and remove from $Q_i^+$. Choose the two edges vertex-disjoint. No copy of a local type contains both, since its order is less than $2r$, so their effects on $R$ add without a mixed term. The two original local vectors lie within $2\delta$ of the corresponding centers, and each subsequent variation is at most $c/16$. Thus
		\[
		\|\Delta R\mp(z_i-z_0)\|\le4\delta+c/8\le c/4.
		\]
		For $x=R(G)$, choose the index and sign in~\eqref{eq:spread} so that
		$\inner{x}{\pm(z_i-z_0)}\le-c\|x\|$.
		It follows that
		\[
		\inner{x}{\Delta R}\le-\tfrac34c\|x\|,
		\qquad \|\Delta R\|\le2R_{\rm sp}+c/4,
		\]
		and hence
		\[
		\|x+\Delta R\|^2
		\le\|x\|^2-\tfrac32c\|x\|+(2R_{\rm sp}+c/4)^2.
		\]
		If $\|x\|\ge\max\{c/4,(2R_{\rm sp}+c/4)^2/c\}$, the right-hand side is at most $(\|x\|-c/4)^2$. Each paired move therefore decreases the norm by at least $c/4$ until it reaches a fixed ball. Starting from $H_0$, this takes at most $B T_0$ paired moves, after enlarging a constant $B=B(r,\varepsilon)$ which also bounds the number of single-edge corrections in the first stage. This budget depends only on the uniform initial and geometric constants.
		
		Choose $K>4B$. Since the first stage uses only candidates from $Q_0^\pm$ and each paired move uses two candidates in total, no candidate class loses more than $3BT_0$ members. Hence every relevant class retains $\Omega_{r,\varepsilon}(T_0)$ unused members throughout the process. After the first edge of a pair has been chosen, at most
		\[
		r\Delta_1(Q)=O_{r,\varepsilon}(1+T_0/n)=o(T_0)
		\]
		members of the other class meet it. An unused disjoint partner therefore always exists, and~\eqref{eq:toggle-budget} follows.
	\end{proof}
	
	\begin{lemma}\label{lem:global-switching-error}
		Assume~\eqref{eq:reservoir-pairs} and~\eqref{eq:toggle-budget}. Then
		\begin{equation}\label{eq:centered-target}
			|S(\Er,H_*)|=D_r(n,p),\qquad
			|S(J,H_*)|=O_{J,\varepsilon}(pn^{v(J)-r})
			\quad(J\ne \Er).
		\end{equation}
		The estimates hold simultaneously for all $J\in\cF_\infty$.
	\end{lemma}
	
	\begin{proof}
		Fix $J\in\cF_\infty$ with $f=v(J)>r$ and $m=e(J)$. Let $\sigma_e=I_{H_*}(e)-I_H(e)$, which belongs to $\{-1,0,1\}$ and is nonzero precisely on $T$. For a fixed copy $J'\cong J$, multilinearity gives
		\[
		\prod_{e\in E(J')}(X_e+\sigma_e)
		=\sum_{A\subseteq E(J')\cap T}
		\left(\prod_{e\in A}\sigma_e\right)
		\left(\prod_{e\in E(J')\setminus A}X_e\right).
		\]
		The empty set gives the original term, and the singletons give the derivatives at $H$. Summing over all copies yields
		\begin{equation}\label{eq:multilinear-switch}
			S(J,H_*)=S(J,H)+\sum_{t\in T}\sigma_tS_t(J,H)+R_{\ge2},
		\end{equation}
		where $R_{\ge2}$ consists of the terms involving at least two toggled edges. The initial event, the exclusion of $\cB$ and~\eqref{eq:toggle-budget} give
		\[
		\begin{aligned}
			|S(J,H)|+\left|\sum_{t\in T}\sigma_tS_t(J,H)\right|
			&\le A_Jp^{m/2}n^{f/2}+|T|L_Jp^{(m-1)/2}n^{(f-r)/2}\\
			&=O_{J,\varepsilon}(p^{m/2}n^{f/2}).
		\end{aligned}
		\]
		
		Each term of $R_{\ge2}$ has absolute value at most one and contains a pair of distinct toggled edges. Assign one such pair to each term by any fixed rule. A fixed copy contributes at most $2^m$ terms, so this costs only a factor depending on $J$. A pair meeting in exactly $a$ vertices occupies $2r-a$ vertices and extends to a copy of $J$ in $O_J(n^{f-2r+a})$ ways when $2r-a\le f$, and in no ways otherwise. For $a\ge1$, the number of pairs with intersection exactly $a$ is at most the number with intersection at least $a$, which is bounded by~\eqref{eq:reservoir-pairs}. Disjoint pairs number at most $|T|^2/2=O_{r,\varepsilon}(pn^r)$. Therefore
		\begin{equation}\label{eq:remainder-bound}
			\begin{aligned}
				|R_{\ge2}|
				&\le C_J\sum_{\substack{0\le a<r\\2r-a\le f}}
				\#\{\{t,s\}\subseteq T:|t\cap s|=a\}\,n^{f-2r+a}\\
				&=O_{J,\varepsilon}(pn^{f-r}).
			\end{aligned}
		\end{equation}
		In particular,
		\begin{equation}\label{eq:unadjusted-statistics}
			|S(J,H_*)|\le C_{J,\varepsilon}
			\left(p^{m/2}n^{f/2}+pn^{f-r}\right)\qquad(f>r).
		\end{equation}
		
		If $f\ge2r$, the absence of isolated vertices gives $m\ge2$, and thus
		$p^{m/2}n^{f/2}\le pn^{f-r}$. If $r<f<2r$ and $J$ is inactive,
		equation~\eqref{eq:inactive-scale} bounds this term by $\sqrt M\,pn^{f-r}$.
		Finally, for an active $J$, the bound on $R(H_*)$ gives
		$|S(J,H_*)|=O_{r,\varepsilon}(d_J)$, and
		\[
		d_J=p^{(m-1)/2}n^{(f-r)/2}
		\le\sqrt p\,n^{(f-r)/2}\le pn^{f-r},
		\]
		using $m\ge2$ and $pn\to\infty$. The one-edge estimate in~\eqref{eq:centered-target} is exactly the nearest-integer edge-count condition in~\eqref{eq:toggle-budget}. The initial and derivative constants were chosen for the whole countable family, with summable failure probabilities. Types with more than $n$ vertices contribute zero. Hence the estimates hold simultaneously for every type.
	\end{proof}

	\begin{proof}[Proof of Theorem~\ref{thm:general-upper}]
		We work first with $p\le1/2$ and fix the constants in the order specified above. For this fixed $M$ and all sufficiently large $n$, Lemma~\ref{lem:candidate-concentration} gives a realization $H$ satisfying the initial, bad-set, candidate and degree bounds prepared above.
		
		Lemmas~\ref{lem:reservoir},~\ref{lem:local-stability},~\ref{lem:paired-descent} and~\ref{lem:global-switching-error} give an $r$-graph $H_*$ satisfying~\eqref{eq:centered-target}. The large-$n$ requirements used to obtain $H_*$ involve only the finite family of types with fewer than $2r$ vertices and the constants just chosen, so the final threshold depends only on $r$ and $\varepsilon$. The constants $A_J,L_J$ for larger $J$ enter only the final constants for the corresponding induced order.
		
		The estimates in~\eqref{eq:centered-target} hold for every $J\in\cF_\infty$, so this single graph $H_*$ can be used for all induced orders. For a fixed $j$, apply~\eqref{eq:centered-expansion} to each $j$-vertex induced type. A centered statistic on $f$ vertices has coefficient $O_{r,j}(n^{j-f})$. Separating the one-edge term in~\eqref{eq:centered-target} gives, simultaneously for $r+1\le j\le n$,
		\[
		u_j^{(r)}(H_*,p)
		\le C_{r,j,\varepsilon}(D_r(n,p)+p)n^{j-r}
		\le3C_{r,j,\varepsilon}\max\{D_r(n,p),p(1-p)\}n^{j-r}.
		\]
		No further lower bound on $n$ is introduced in this last step; only the constants may depend on the chosen order $j$.
		For $p>1/2$, apply the construction at density $1-p$ and complement the resulting $r$-graph. Complementation preserves the nearest-integer edge discrepancy and permutes the induced types at every order. This proves the theorem.
	\end{proof}

	\section*{Acknowledgments}
	The authors thank Shuping Huang for helpful suggestions. They also thank Eureka for helpful comments on the presentation.
	
	\smallskip
	\noindent\textbf{Funding.}
	This work was supported by grants 2023YFA1010202 from the National Key
	R\&D Program of China and 12371342 from the National Natural Science
	Foundation of China.
	
	\section*{Declaration on the use of generative AI}
	During the preparation of this manuscript, the authors used ChatGPT (OpenAI) for language editing and consistency checks in exposition. The authors reviewed all resulting text and are responsible for all mathematical content.


\begingroup
	\small
	\begin{thebibliography}{99}
		\setlength{\itemsep}{1pt}
		
		\bibitem{BS2006}
		B. Bollob\'as and A. D. Scott,
		Discrepancy in graphs and hypergraphs,
		in \emph{More Sets, Graphs and Numbers}, Bolyai Soc. Math. Stud. 15,
		Springer, Berlin, 2006, 33--56.
		\doi{10.1007/978-3-540-32439-3_3}.
		
		\bibitem{BS2011}
		B. Bollob\'as and A. D. Scott,
		Intersections of graphs,
		\emph{J. Graph Theory} 66 (2011), 261--282.
		\doi{10.1002/jgt.20489}.
		
		\bibitem{BS2015}
		B. Bollob\'as and A. D. Scott,
		Intersections of hypergraphs,
		\emph{J. Combin. Theory Ser. B} 110 (2015), 180--208.
		\doi{10.1016/j.jctb.2014.08.002}.
		
		\bibitem{CHZ2025}
		Y. Chen, S. Huang and Q. Zeng,
		Graded discrepancy of graphs and hypergraphs,
		\arxiv{2505.21690} (2025).
		
		\bibitem{CW2026}
		Y. Chen and P. Wang,
		Minimum graded discrepancy of graphs,
		\emph{Ann. Appl. Math.} 42 (2026), no.~1, 94--102.
		\doi{10.4208/aam.oa-2025-0039}.
		
		\bibitem{EGPS1988}
		P. Erd\H{o}s, M. Goldberg, J. Pach and J. Spencer,
		Cutting a graph into two dissimilar halves,
		\emph{J. Graph Theory} 12 (1988), 121--131.
		\doi{10.1002/jgt.3190120113}.
		
		\bibitem{ES1971}
		P. Erd\H{o}s and J. Spencer,
		Imbalances in $k$-colorations,
		\emph{Networks} 1 (1971), 379--385.
		\doi{10.1002/net.3230010407}.
		
		\bibitem{Janson1995}
		S. Janson,
		A graph Fourier transform and proportional graphs,
		\emph{Random Structures Algorithms} 6 (1995), 341--351.
		\doi{10.1002/rsa.3240060221}.
		
		\bibitem{JK1991}
		S. Janson and J. Kratochv\'il,
		Proportional graphs,
		\emph{Random Structures Algorithms} 2 (1991), 209--224.
		\doi{10.1002/rsa.3240020205}.
		
		\bibitem{JS1992}
		S. Janson and J. Spencer,
		Probabilistic construction of proportional graphs,
		\emph{Random Structures Algorithms} 3 (1992), 127--137.
		\doi{10.1002/rsa.3240030203}.
		
		\bibitem{Karrman1993}
		J. K\"arrman,
		Existence of proportional graphs,
		\emph{J. Graph Theory} 17 (1993), 207--220.
		\doi{10.1002/jgt.3190170209}.
		
		\bibitem{LLTY2025}
		D. Luong-Le, T. Tran and D. Yang,
		Relative discrepancy of hypergraphs,
		\arxiv{2506.23264} (2025).
		
		\bibitem{MNS2015}
		J. Ma, H. Naves and B. Sudakov,
		Discrepancy of random graphs and hypergraphs,
		\emph{Random Structures Algorithms} 47 (2015), 147--162.
		\doi{10.1002/rsa.20522}.
		
		\bibitem{NPS2018}
		H. Naves, O. Pikhurko and A. Scott,
		How unproportional must a graph be?,
		\emph{European J. Combin.} 73 (2018), 138--152.
		\doi{10.1016/j.ejc.2018.05.007}.
		
		\bibitem{RST2026}
		E. R\"aty, B. Sudakov and I. Tomon,
		Positive discrepancy, MaxCut, and eigenvalues of graphs,
		\emph{Trans. Amer. Math. Soc.} 379 (2026), no.~3, 2111--2140.
		\doi{10.1090/tran/9551}.
		
		\bibitem{RT2026}
		E. R\"aty and I. Tomon,
		Bisection width, discrepancy, and eigenvalues of hypergraphs,
		\emph{J. Combin. Theory Ser. B} 177 (2026), 186--215.
		\doi{10.1016/j.jctb.2025.11.003}.
		
		\bibitem{SSZ2024}
		A. Sah, M. Sawhney and D. G. Zhu,
		Local limit theorem for joint subgraph counts,
		\emph{Ann. Probab.}, to appear.
		\arxiv{2412.09535}.
		
	\end{thebibliography}
	\endgroup
\end{document}